\documentclass[a4paper,12pt,reqno]{amsart}
\usepackage{verbatim}
\usepackage{enumitem}
\usepackage{hyperref}
\usepackage{amssymb}
\usepackage[all]{xy}
\usepackage{tikz}
\usetikzlibrary{positioning,arrows,cd}

\newcommand{\Q}{\mathbb{Q}}
\newcommand{\G}{\Gamma}

\newcommand{\R}{\mathbb{R}}
\newcommand{\Z}{\mathbb{Z}}

\newtheorem{theorem}{Theorem}[section]
\newtheorem{lemma}[theorem]{Lemma}
\newtheorem{corollary}[theorem]{Corollary}
\newtheorem{proposition}[theorem]{Proposition}

\newtheorem{definition}[theorem]{Definition}

\newtheorem{question}[theorem]{Question}

\newtheorem{thmx}{Theorem}

\newtheorem{corx}[thmx]{Corollary}
\theoremstyle{remark}
\newtheorem{remark}[theorem]{Remark}

\usepackage{colortbl}
\DefineNamedColor{named}{RoyalBlue}     {cmyk}{1,0.50,0,0}
\DefineNamedColor{named}{BrickRed}      {cmyk}{0,0.89,0.94,0.28}

\begin{document}

\title{Finite sets containing zero are mapping degree sets}

\author[C.\ Costoya]{Cristina Costoya}
\address{CITMAga, Departamento de Matem\'aticas,
Universidade de Santiago de Compostela, 15782-Santiago de Compostela, Spain.}
\email{cristina.costoya@usc.es}

\author[V.\ Mu\~{n}oz]{Vicente Mu\~{n}oz}
\address{Instituto de Matem\'atica Interdisciplinar and
Departamento de \'Algebra, Geometr\'{\i}a y Topolog\'{\i}a, Universidad Complutense de Madrid, Plaza de las Ciencias, 3, 28040-Madrid, Spain}
\email{vicente.munoz@ucm.es}

\author[A.\ Viruel]{Antonio Viruel}
\address{Departamento de \'Algebra, Geometr\'{\i}a y Topolog\'{\i}a, Universidad de M\'alaga,
Campus de Teatinos, s/n, 29071-M\'alaga, Spain}
\email{viruel@uma.es}


\begin{abstract}
In this paper we solve in the positive  the question of whether any finite set of integers, containing $0$, is  the mapping degree set between two oriented closed connected manifolds of the same dimension. We extend this question to the rational setting, where an affirmative answer is also given.
\end{abstract}

\maketitle

%%%%%%%%%%%%%%%%%%%%%%%%%%%%%%%%%%%%%%%%%%%%%%%%%%%%%
%%%%%%%%%%%%%%%%%%%%%%%%%%%%%%%%%%%%%%%%%%%%%%%%%%%%%
\section{Introduction}\label{sec:intro}
%%%%%%%%%%%%%%%%%%%%%%%%%%%%%%%%%%%%%%%%%%%%%%%%%%%%%
%%%%%%%%%%%%%%%%%%%%%%%%%%%%%%%%%%%%%%%%%%%%%%%%%%%%%

In this paper, we settle in the positive various questions which have been raised about $D(M,N)$, the set of mapping degrees between two oriented closed connected manifolds $M$ and $N$ of the same dimension: $$D(M,N)=\{ d \in \Z\, |\, \exists f:M\to N, \, \deg(f)=d\}.$$

C.\ Neofytidis, S.\ Wang, and Z.\ Wang  \cite[Problem 1.1]{NWW} discuss the problem of finding, for every set $A \subset \mathbb Z$ containing $0$, two oriented closed connected manifolds $M$ and $N$ of the same dimension such that $A=D(M,N)$. Note that $0 \in A$ is a necessary  condition  as the constant map $M \rightarrow N$ is of degree zero.

 A cardinality argument  shows that when $A$ is an infinite set,  the problem above is solved in the negative \cite[Theorem 1.3]{NWW}. Indeed,  there are uncountably many infinite subsets  of $\mathbb Z$ containing $0$, compared to the  countably many  mapping degree sets that exist for pairs of oriented closed connected manifolds with the same dimension.  Hence, \emph{most} of the infinite sets are not realizable as mapping degree sets.  Now, using a computability argument, C.\ L{\"o}h and M.\ Uschold \cite[Proposition A.1]{LohUsc} prove that $D(M,N)$ is  a recursively enumerable set. This provides  a sufficient condition for infinite sets of integers (containing $0$)  not being realizable as mapping degree sets (see \cite[Example A.2]{LohUsc}).

Thus, one might ask:

\begin{question}[{\cite[Problem 1.4]{NWW}}] \label{mainquestion}
Let $A$ be a finite set of integers containing $0$. Is $A=D(M,N)$  for some oriented closed connected  manifolds $M, N$ of the same dimension?
\end{question}

\begin{remark}\label{rem:inflexible}
It is important to notice that
 if $\{0\} \subsetneq A= D(M,N)$, $A$ finite,
for some manifolds $M$ and $N$,  then $D(M, M)$ and $D(N,N)$ must both  be contained in $\{0, 1, -1\}$. Otherwise, if there exists $g\colon M \rightarrow M$ with $|\deg(g)| >1,$ then for every  non-zero degree $f\colon  M \rightarrow N$  (which exists by assumption),  the subset  $\{ \deg(f \circ g^m) \mid m \in \mathbb N \} $  of $D(M, N)$ is unbounded. This leads to a contradiction as $A=D(M, N)$ is finite.  The same  follows for $D(N,N)$.
\end{remark}

An  oriented closed connected manifold $M$ satisfying $D(M, M) \subset  \{0, 1, -1\}$ is called an \emph{inflexible} manifold \cite[Definition 1.4]{CL}.   This condition is equivalent to asking  that $D(M,M)$ is bounded: since it is a multiplicative semi-group, if there exists any $\ell\in D(M,M)$ with $|\ell|>1$, then $D(M,M)$ is unbounded.  Simply connected inflexible manifolds are rare objects that have appeared quite  recently in literature  using rational homotopy theory and surgery theory (see  \cite{CMuV1} for an account on the simply connected inflexible manifolds that are known at present). Not surprisingly,  and in light of Remark \ref{rem:inflexible}, part of our key constructions will use rational homotopy methods.

The main result in this work answers Question \ref{mainquestion} positively:

\begin{thmx}\label{thm:main-3-manifolds}
Let $A$ be a finite set of integers containing $0$.  Then, $A= D(M,N)$ for some oriented closed connected $3$-manifolds $M,N$.
\end{thmx}

The proof of this theorem will be carried out at the end of Section \ref{sec:proofmainthm}. Appealing to  \cite[Example 1.5]{NWW}, we point out that the $3$-dimension of the manifolds is the lowest possible.

A second problem related to Question \ref{mainquestion} is also treated in this paper. More precisely, let the  \emph{rational mapping degree set} between oriented closed connected $n$-manifolds $M, N$ be the following set
\begin{gather*}
 D_\Q(M,N)=\{ d \in \Q\, |\, \exists f\colon(M_{(0)},  [M]_\Q )\to (N_{(0)}, [N]_\Q ), \, \deg(f)=d\},
 \end{gather*}
where $[M] \in H^n (M; \mathbb Z)$  denotes the cohomological fundamental class of $M$,  $[M]_\Q \in H^n(M; \Q) $ denotes the rational cohomological fundamental class  of $M$, and $M_{(0)}$ the rationalization of $M$. 
Note that we must specify the choice of fundamental class for $M_{(0)}$ in order to define the degree of a map since this class is not determined by the space $M_{(0)}$. In other words, different choices of manifolds $M$ can result in the same space $M_{(0)}$ with  different fundamental classes.

Then, we raise the following question, which can be thought of as a rational version of \cite[Problem 1.4]{NWW}:

  \begin{question}\label{mainquestionrational}
Let $A$ be a finite set of rational numbers containing $0$. Is $A=D_\Q(M,N)$  for some oriented closed connected manifolds $M, N$ of the same dimension?
\end{question}
In Section \ref{sec:rational-mappingdegree} we solve this problem in the positive by proving:

\begin{thmx}\label{thm:main-rational}
Let  $A$ be a finite set of rational numbers containing $0$. Then  $A=D_\Q(M,N)$ for some oriented closed connected manifolds $M,N$. Moreover, given any integer $k\geq 47$, the manifolds $M,N$ above can be chosen $k$-connected.
\end{thmx}

The proofs of  Theorem \ref{thm:main-3-manifolds} and Theorem \ref{thm:main-rational} consist primarily of two main steps:

\begin{itemize}
\item \emph{Arithmetical decomposition of finite sets: } In Section \ref{sec:combinatoric} we demonstrate how to decompose the candidate $A$ to be realized as the mapping degree set of manifolds, as an intersection of   sums   over specifically designed sequences $S_{B_i},$ $i =0, \ldots, n,$ of integers  (see Definition \ref{def:sumover}). Each of those sums gradually approaches $A$ (Proposition \ref{prop:clave-aritmetica}, Corollary \ref{cor:clave-aritmetica-racional}).

\smallskip

\item \emph{Spherical fibrations:} In Sections \ref{sec:degree_connected}  and \ref{sec:rational-mappingdegree}, we use  certain inflexible manifolds  (respectively, inflexible Sullivan algebras) as the basis of spherical fibrations  where the total spaces are also inflexible manifolds (respectively, inflexible Sullivan algebras).  The relations between connected sums and mapping degree sets (see Propositions \ref{prop:iteratedsums}, \ref{prop:3_rational}, and \ref{prop:iteratedsums_simply_connected}) allow us to consider iterated connected sums of the total spaces, first to realize the sums $S_{B_i}$ above mentioned, and subsequently to realize the candidate $A$.
 \end{itemize}

Looking at the connectivity,  while manifolds from Theorem  \ref{thm:main-rational} are simply connected (indeed, they are as highly connected as desired) the ones from Theorem \ref{thm:main-3-manifolds} have non-trivial fundamental group.  In Section \ref{sec:adamsoperations} we will use unstable Adams operations to prove the following results that guarantee that manifolds realizing finite sets of integers can be chosen simply connected:

\begin{thmx}\label{thm:main-many-dimensions-manifolds}
Suppose that there exists an oriented closed $k$-connected $2m$-manifold $\Sigma$, $m>1$, satisfying that $\Sigma_{(0)}$ is inflexible and $\pi_j(\Sigma_{(0)})=0$ for $j\geq 2m-1$. Then any finite set of integers $A$ containing $0$ can be realized as $A=D(M,N)$ for some oriented closed $k$-connected $(4m-1)$-manifolds $M, N$.
\end{thmx}

Examples of simply connected  manifolds fulfilling the hypotheses of Theorem \ref{thm:main-many-dimensions-manifolds} can be found in \cite[Example 3.8]{Amann}, \cite[Examples 5.1 and 5.2]{AL} and  \cite[Theorem 6.8, Theorem II.5]{CL} (see also \cite[Theorem 1.4]{NeoProd}). Hence, the following holds:

\begin{corx}\label{cor:main-simply connected-manifolds}
Any finite set of integers $A$ containing $0$ can be realized as $A=D(M,N)$ for some  oriented closed simply connected manifolds $M, N$ of the same dimension.
\end{corx}

Also, we would like to point out that in every dimension $n>6$, there exist infinitely many nilmanifolds $\Sigma$ satisfying that $\Sigma_{(0)}$ is inflexible, as it follows from \cite{AncCamp}, \cite{Bel} and \cite{Has}.  Hence, we have the following:

\begin{corx}\label{cor:nilmanifolds}
For every $m>3$, any finite set of integers $A$ containing $0$ can be realized as $A=D(M,N)$ for some oriented closed connected $(4m-1)$-manifolds $M, N$.
\end{corx}

Recall that an oriented closed connected manifold is \emph{strongly chiral}  if it does not admit  self-maps of degree $-1$ (see \cite{MullnerAGT}). Also recall that a subset $A \subset \mathbb Z$ is said \emph{symmetric} if $A=-A$ where $-A = \{ -a \,|\, a \in A\}$. We finish the introduction by mentioning that our results provide us with a method for obtaining strongly chiral manifolds:

\begin{remark}\label{rem:symmetric}
Let $A$ be a finite set of integers strictly containing $0$.  Theorem \ref{thm:main-3-manifolds} and  Theorem \ref{thm:main-many-dimensions-manifolds} illustrate how to construct  oriented closed connected manifolds  $M, N$  satisfying that  $A = D(M,N)$.  Those manifolds need to be inflexible, as mentioned in Remark \ref{rem:inflexible}.  Moreover,  if  we choose a {non-symmetric} set $A$, then $M$ and $N$ need also to be strongly-chiral manifolds. Otherwise,  if $-1 \in D(M,M)$ (respectively, $-1 \in D(N,N)$),  since  $D(M,M)$ (respectively, $D(N,N)$) acts on $A=D(M,N)$  by multiplication on the left (respectively, right),  the set $A$ would be symmetric, leading thus to a contradiction.   
\end{remark}

\begin{remark}\label{rem:gen_comment_1}
Most of our arguments are homotopic in nature, which blurs the distinction between a map and the homotopy class it represents. Consequently, most of our  diagrams are commutative up to homotopy. However, it is important to note that the manifolds under consideration are smooth; hence, we often substitute maps with smooth ones within the same homotopy class.\end{remark}

\subsection*{Funding}
This work was partially supported by MCIN/AEI/10.13039/501100011033 [PID2020-
115155GB-I00 to C.C., PID2020-118452GB-I00 to V.M., and PID2020-118753GB-I00 to A.V.]. The third author was also partially supported by PAIDI 2020 (Andalusia) grant PROYEXCEL-00827.

\subsection*{Acknowledgments}
We would like to express our sincere gratitude to C.\ Neofytidis for pointing out an error in a previous version of this manuscript. We also thank the referee for a thorough reading and many useful suggestions.

%%%%%%%%%%%%%%%%%%%%%%%%%%%%%%%%%%%%%%%%%%%%%%%%%%%%%
%%%%%%%%%%%%%%%%%%%%%%%%%%%%%%%%%%%%%%%%%%%%%%%%%%%%%
\section{Some arithmetic combinatorics}\label{sec:combinatoric}
%%%%%%%%%%%%%%%%%%%%%%%%%%%%%%%%%%%%%%%%%%%%%%%%%%%%%
%%%%%%%%%%%%%%%%%%%%%%%%%%%%%%%%%%%%%%%%%%%%%%%%%%%%%

In this section we show that   every  finite set $A \subset \mathbb Z$  (respectively, $\subset \mathbb Q$)  containing $0$  can be expressed  as the intersection of sums over certain sequences of integers, that gradually approach $A$. The sequences have an additional property (see Proposition \ref{prop:clave-aritmetica})  that will be crucial to prove Theorem \ref{thm:main-many-dimensions-manifolds} in Section \ref{sec:adamsoperations} below.

\begin{definition}\label{def:sumover}
Let $B=(b_i)_{i\in I}$ be a finite sequence of integers (respectively, rational numbers). We write
\begin{equation*}
S_B:=\sum_{i\in I}\{0,b_i\}\subset \Z\,\text{(respectively, $\subset \Q$)},
\end{equation*}
and we refer to it as the sum over the sequence $B.$
\end{definition}

\begin{proposition}\label{prop:clave-aritmetica}
Let $d_1,\ldots,d_n$ be pairwise distinct non-zero integers. For every positive integer $m\geq 1$, there exist finite sequences  $B(i)$, $i=0,\ldots, n$,  of non-zero integers, such that
 \begin{equation*}
 \{0, d_1,\ldots, d_n\}=\bigcap_{i=0}^n  S_{B(i)},
 \end{equation*}
and such that every element in $B(i)$ can be written as a power $\pm k^{m}$ for some positive integer $k$ coprime to $m!$.
\end{proposition}

\begin{proof}
Fix $m \geq 1$. Since the construction of $B(i)$, $i=0, \ldots, n,$ depends on the sign of the pairwise distinct $d_i \in \mathbb Z$, $i=1, \ldots, n$, we  write them as an ordered sequence
$$
\{-a_r<\ldots<-a_1<0<e_1<\ldots <e_s\}
$$
where $n=r+s$. We assume $a_0=0=e_0$.

In the first step,  let $B(0)$ be the sequence consisting of $a_r$ copies of $-1=-1^{m}$ and $e_s$ copies of $1=1^{m}$. Thus $$\{-a_r<\ldots <e_s\}\subset S_{B(0)}=[-a_r,e_s]\cap\Z.$$
In the second step, for $j=1,\ldots, s$, choose a positive $k_j\in\Z$ coprime with $m!$ such that $k_j^m>\max\{e_s,e_j+a_r\}$. Then, let  $B(j)$  be the sequence consisting of $k_j^m-e_j$ copies of $-1=-1^m$, $e_{j-1}$ copies of $1=1^m$, and one copy of $k_j^m$. Hence,
 $$
 \{-a_r<\ldots <e_s\}\subset S_{B(j)}=\big([-(k_j^m-e_j),e_{j-1}]\cup [e_j, k_j^m+e_{j-1}]\big)\cap\Z.
 $$
Finally,  for $j=s+1,\ldots, n$, choose a positive $k_j\in\Z$ coprime with $m!$ such that $k_j^m>\max\{a_r, \, a_{j-s}+e_s\}$. Then, let  $B(j)$  be the sequence consisting of $k_j^m-a_{j-s}$ copies of $1=1^m$,  $a_{j-s-1}$ copies of $-1=-1^m$, and one copy of  $-k_j^m$.  Hence,
$$
\{-a_r<\ldots <e_s\}\subset S_{B(j)}=\big([-k_j^m-a_{j-s-1},-a_{j-s}]\cup [-a_{j-s-1}, k_j^m-a_{j-s-1}]\big)\cap\Z.
$$

 All of the above implies that
 $$
 \{0, d_1,\ldots, d_n\}=\{-a_r<\ldots <e_s\}=\bigcap_{i=0}^n  S_{B(i)}.
 $$
\end{proof}

For $A\subset \Q$ and $\lambda\in\Q$ we write
 $$
 \lambda A:=\{\lambda a\,|\, a\in A\}.
  $$
Notice that if $B(i)$ is a finite sequence of not necessarily pairwise distinct non-zero rational numbers, $i=0,\ldots,n$,  for every  $\lambda\in\Q$, we have that
 $$
 \lambda\big(\bigcap_{i=0}^n  S_{B(i)}\big)=\bigcap_{i=0}^n  S_{\lambda B(i)}.
 $$
 Therefore, the following is a direct consequence of Proposition \ref{prop:clave-aritmetica}:

\begin{corollary}\label{cor:clave-aritmetica-racional}
Let $d_1,d_2,\ldots,d_n$ be pairwise distinct non-zero rational numbers. Then, there exist finite sequences  $B(i)$, $i=0,\ldots, n$,  of non-zero rational numbers, such that
 \begin{equation*}
 \{0, d_1,\ldots, d_n\}=\bigcap_{i=0}^n  S_{B(i)}.
 \end{equation*}
\end{corollary}

%%%%%%%%%%%%%%%%%%%%%%%%%%%%%%%%%%%%%%%%%%%%%%%%%%%%%
%%%%%%%%%%%%%%%%%%%%%%%%%%%%%%%%%%%%%%%%%%%%%%%%%%%%%
\section{Circle bundles over inflexible $2$-manifolds: mapping degree set}\label{sec:degree_connected}
%%%%%%%%%%%%%%%%%%%%%%%%%%%%%%%%%%%%%%%%%%%%%%%%%%%%%
%%%%%%%%%%%%%%%%%%%%%%%%%%%%%%%%%%%%%%%%%%%%%%%%%%%%%

This section is devoted to prove Theorem  \ref{thm:main-3-manifolds}. As explained in the introduction (see Remark \ref{rem:inflexible}) if we want to realize a finite set of integers, strictly containing $0$, as  a mapping degree set $D(M, N)$, then  both $M$ and $N$ need to be inflexible manifolds. We are going to consider circle bundles  over certain inflexible $2$-manifolds, with prescribed Euler class, whose total space is again an inflexible $3$-manifold. These $3$-manifolds will be used as building blocks to construct, by means of iterated connected sums, manifolds $M$ and $N$.

We first collect a couple of results that are needed. The first one is immediately  obtained by iterating \cite[Lemma 7.8]{CMuV1}, \cite[Lemma 3.5]{NWW}: 
\begin{lemma}\label{lem:lemas_CL+NWW}
Let $M_i, \, i=1, \ldots, k, $ and $N$ be oriented closed connected $n$-manifolds. Then
$$
\sum\limits_{i=1}^k D(M_i,N) \subset D(\operatornamewithlimits{\#}_{i=1}^k M_i,N)
$$
 Moreover, if  $\pi_{n-1}(N)=0$, then
$$
\sum\limits_{i=1}^k D(M_i,N) =D(\operatornamewithlimits{\#}_{i=1}^k M_i,N)
$$
\end{lemma}

Reformulating and iterating \cite[Lemma 4.3]{NWW}, we get the following:
\begin{lemma}\label{lem:lema_NWW}
Let $M$ and $N_i,\, i=1, \ldots, k,$ be oriented closed connected manifolds of the same dimension. Then
$$D(M,\operatornamewithlimits{\#}_{i=1}^k N_i)\subset \operatornamewithlimits{\cap}_{i=1}^k D(M,N_i).$$
\end{lemma}

Whenever we need to enlarge the fundamental group of our manifolds to construct mapping degree sets we shall use handle-bodies:

\begin{definition}
    Given integers $n,k>0$, we denote by $H(n,k)$ the oriented closed connected $n$-manifold arising from the $k$-fold connected sum of oriented handle bodies $S^{n-1}\times S^1$, that is
    $$H(n,k):={\operatornamewithlimits{\#}_{i=1}^k} (S^{n-1}\times S^1).$$
\end{definition}

\begin{lemma}\label{lem:fix_3.3}
    Let $n>2$ and $N$ be an oriented closed connected $n$-manifold such that $\pi_{n-1}(N)=0$. Then $D\big(H(n,k),N\big)=\{0\},$ for every integer $k>0.$      
\end{lemma}
\begin{proof}
Given a space $X$, and a positive integer $m$, let $X^{(m)}$ and $X\langle m\rangle$ respectively denote the $m$th Postnikov stage and the $m$-connected cover of $X$.

Since Postnikov stages and connected covers can be constructed in a functorial way \cite[Example 1.A.1.1]{Dror}, given $f\colon H(n,1)\to N$ we can take the $(n-2)$th Postnikov stage and the $(n-2)$-connected cover to obtain the following commutative diagram
\begin{equation}\label{eq:fix_3.3}
\begin{tikzcd}
\makebox[0pt][r]{$S^{n-1}\simeq\,$}H(n,1)\langle n-2\rangle \arrow[r] \arrow[d, "f\langle n-2\rangle"] & H(n,1) \arrow[r] \arrow[d, "f"] & H(n,1)^{(n-2)}\makebox[0pt][l]{$\,\simeq S^1$} \arrow[d, "f^{(n-2)}"]\\
N\langle n-2\rangle \arrow[r] & N \arrow[r]  & N^{(n-2)}.
\end{tikzcd}   
\end{equation}
We now compare the homological Serre spectral sequences (Sss) associated to the two fiber sequences above. On the one hand, the Sss associated to the top-row fiber sequence in Equation \eqref{eq:fix_3.3} collapses at the $E^2_{*,*}$-page and the fundamental class $[H(n,1)]\in H_n\big(H(n,1)\big)$ is represented by $[S^{n-1}]\otimes [S^1]$, where $[M]$ now denotes  the \emph{homological} fundamental class of $M$. Therefore, using the edge morphisms, we get that the class $H_n(f)\big([H(n,1)]\big)\in H_n(N)$ is represented by $H_{n-1}(f\langle n-2\rangle)\big([S^{n-1}]\big)\otimes H_1(f^{(n-2)})\big([S^1]\big)$ in the Sss associated to the bottom-row fiber sequence in Equation \eqref{eq:fix_3.3}.

On the other hand, since $\pi_{n-1}(N)=0,$ then $N\langle n-2\rangle$ is indeed $(n-1)$-connected and therefore $H_{n-1}(N\langle n-2\rangle)=0.$  Thus $H_{n-1}(f\langle n-2\rangle)$ is the trivial morphism,  and the class $H_{n-1}(f\langle n-2\rangle)\big([S^{n-1}]\big)\otimes H_1(f^{(n-2)})\big([S^1]\big)$ is trivial. In other words $D\big(H(n,1),N\big)=\{0\}$.

Finally, for the general case, we apply Lemma \ref{lem:lemas_CL+NWW}, to $H(n,k)={\operatornamewithlimits{\#}_{i=1}^k}H(n,1)$ and $N$ to get that
    $$D\big(H(n,k),N\big)=\sum\limits_{i=1}^k D\big(H(n,1),N\big)=\{0\}.$$
\end{proof}

A key result in our arguments is the following:

\begin{lemma}\label{lem:3}
Let $M_i$ and $N_i$, $i=1,\ldots, s,$ be oriented closed connected $n$-manifolds, $n>2$, and let $A\subset \operatornamewithlimits{\cap}_{i=1}^s D(M_i,N_i)$ be a finite set containing $0$. Then there exists an integer $\ell\geq 0$ such that
$$A\subset D\big((\operatornamewithlimits{\#}_{j=1}^s M_j)\# H(n,\ell), \operatornamewithlimits{\#}_{i=1}^s N_i\big).$$
\end{lemma}

\begin{proof}
Since $H(n,\ell_1)\# H(n,\ell_2)=H(n,\ell_1+\ell_2),$ and $$A\subset \bigcap_{i=1}^s D(M_i,N_i)=\big(\bigcap_{i=1}^{s-1} D(M_i,N_i)\big)\cap D(M_r,N_r),$$ an easy inductive argument reduces the proof to the case $s=2$.

Let us define the integers 
$$\{0,d_1,\ldots, d_r\}:=A\subset D(M_1,N_1)\cap D(M_2,N_2),$$
where $d_1, \ldots, d_r$ are pairwise distinct and non-zero.
Since every continuous map from $ M_i$ to  $N_i$, $i=1,2,$ is homotopic to a smooth map, for each 
$$d_j\in D(M_1,N_1)\cap D(M_2,N_2)\subset D(M_i,N_i),  \;  i =1,2$$ 
we can choose  $f_{i,j}\colon M_i\to N_i$ such that $f_{i,j}$ is smooth, and $\deg(f_{i,j})=d_j.$ Then, since the set of singular values of a smooth map has measure zero according to Sard's Theorem \cite{Sard},   we can choose a regular value $y_{i,j}\in N_i$ for $f_{i,j}$ such that $y_{i,j}\ne y_{i,j'},$ and $f_{i,j}^{-1}(y_{i,j})\cap f_{i,j'}^{-1}(y_{i,j'})=\emptyset,$ for $j\ne j'.$  Let us set $$f_{i,j}^{-1}(y_{i,j})=:\{x_{i,j,1},\ldots, x_{i,j,s_{i,j}}\},$$ where $s_{i,j}\geq |d_j|.$

Now, by the Stack of Records Theorem \cite[Theorem 9.1]{SoRT}, \cite[p.\ 26]{Guillemin-Pollack}, there exists an $n$-dimensional open disc $U_{i,j}\subset N_i$ such that $y_{i,j}\in U_{i,j}$ and $$f_{i,j}^{-1}(U_{i,j})=\bigcup\limits_{k=1}^{s_{i,j}} {V_{i,j,k}}$$ where $x_{i,j,k}\in V_{i,j,k}\subset M_i$, ${V_{i,j,k}}\cap {V_{i,j,k'}}=\emptyset$ for $k\ne k',$  and $f_{i,j}|_{V_{i,j,k}}\, \colon {V_{i,j,k}}\to {U_{i,j}}$ is a diffeomorphism. Therefore $V_{i,j,k}\subset M_i$ is an open disc, $\deg(f_{i,j}|_{V_{i,j,k}})=\pm 1$, and
$$\deg(f_{i,j})=\sum_{k=1}^{s_{i,j}} \deg(f_{i,j}|_{V_{i,j,k}}).$$
We may assume that $\deg(f_{i,j}|_{V_{i,j,k}})=\operatorname{sign}(\deg(f_{i,j})),$ for $k=1,\ldots, |d_j|$ while the following $s_{i,j}-|d_j|$ diffeomorphisms, which are on even number, have alternating signs.

Since we are dealing with compact Hausdorff spaces, we can shrink all the discs and assume that the closures $\overline{U_{i,j}}\subset N_i$, $j=1,\ldots, r,$ are pairwise disjoint, as well as the closures $\overline{V_{i,j,k}}\subset M_i$, $k=1,\ldots s_{i,j},$
and $f_{i,j}|_{\overline{V_{i,j,k}}}\, \colon{\overline{V_{i,j,k}}}\to \overline{U_{i,j}}$ are diffeomorphisms. We fix diffeomorphisms $g_j\colon \partial U_{1,j}\to \partial U_{2,j}$ for $j=1,\ldots, r.$

Let $M$ be the $n$-manifold obtained by removing discs and adding handles $S^{n-1}\times [0,1]$ to the disjoint union of $M_1$ and $M_2$ as follows:
\begin{enumerate}
    \item First, we remove all discs $V_{i,j,k}\subset M_i,$ $i=1,2,$ from the disjoint union of $M_1$ and $M_2$. Let $K_0$ denote the disconnected $n$-manifold with boundary that we obtain;
    \item Then, for each $k=1,\ldots, |d_j|$, we add to $K_0$ a handle, that we label $(d_j,k,0)$, by identifying $S^{n-1}\times\{0\}$ with $\partial(V_{1,j,k})$ and $S^{n-1}\times\{1\}$ with $\partial(V_{2,j,k})$ such that if $(x,t)\in S^{n-1}\times I$ is in the handle $(d_j,k,0)$, then $g_j\big(f_{1,j}(x,0)\big)=f_{2,j}(x,1).$ Thus we obtain $K_1$, a connected $n$-manifold, with boundary when $s_{i,j} > |d_j|$; 
    \item Finally, for each $i=1,2$, $j=1,\ldots, r$, if $s_{i,j} > |d_j|$, for $m=1,\ldots, \frac{s_{i,j}-|d_j|}{2}$, we add to $K_1$ a handle, that we label $(d_j,m,i)$, identifying $S^{n-1}\times\{0\}$ with $\partial(V_{i,j,|d_j|+2m-1})$ and $S^{n-1}\times\{1\}$ with $\partial(V_{i,j,|d_j|+2m})$, such that if $(x,t)\in S^{n-1}\times I$ is in the handle $(d_j,m,i)$, then $f_{i,j}(x,0)=f_{i,j}(x,1).$ 
\end{enumerate}
Let $M$ be the resulting manifold. Observe that $M=M_1\# M_2\# H(n,\ell)$ for some $\ell\geq 0.$ We are going to prove that $$\{0,d_1,\ldots, d_r\}\subset D(M,N_1\#N_2).$$

Indeed, given any $d_j$, for $j=1,\ldots, r,$ we construct a map $F_j\colon M\to N_1\#N_2$ of degree $d_j$ as follows:
\begin{enumerate}
    \item We apply the pinching-stretching-collapsing procedure (see Figure \ref{fig:pinching-stretching-collapsing}) on all handles within $M$ labelled as $(d_{j'},k,0)$ and $(d_{j'},m,i)$ for $j'\ne j$. Consequently, we obtain a map  $M \to \widetilde{M}_j$ where $\widetilde{M}_j$ can be identified as the connected manifold derived from the disjoint union of $M_1$ and $M_2$ where discs $V_{i,j,k}\subset M_i,$ $i=1,2,$ are removed and handles with labels $(d_j,k,0)$ and $(d_j,m,i)$ are added as above.

    \item Let $N_1\#N_2$ be the manifold obtained by first removing the discs $U_{i,j}\subset N_i$, $i=1,2$ and then by identifying $\partial(U_{1,j})$ with $\partial(U_{2,j})$ using the previously introduced diffeomorphism $g_j$.
Then, define $\widetilde{M}_j\to N_1\# N_2$ by:
    \begin{itemize}
        \item  $(x,t)\in S^{n-1}\times I$ in the handle $(d_j,k,0)$ is mapped onto  $g_j\big(f_{1,j}(x,0))=f_{2,j}(x,1),$

        \item $(x,t)\in S^{n-1}\times I$ in the handle $(d_j,m,i)$ is mapped onto $f_{i,j}(x,0)=f_{i,j}(x,1),$ 

        \item a point in $M_i$ is mapped by $f_{i,j}$ for $i=1,2$. 
    \end{itemize}

    \item Let $F_j$ be the composition of the maps $M\to\widetilde{M}_j$ and $\widetilde{M}_j\to N_1\# N_2$ constructed above. Then, $\deg(F_j) = d_j$. 
\end{enumerate}

%%%%
% Figura del pinching-stretching-collapsing
%%%%

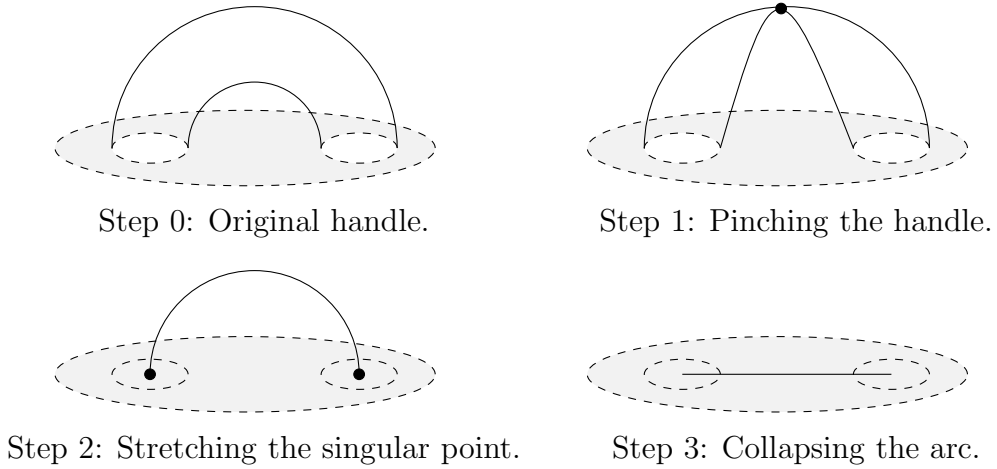
\begin{figure}
\centering
\begin{tikzpicture}
% Asa
\filldraw[color=black,fill=black!5, dashed] (2.25,0) ellipse (2.5 and 0.5);
\filldraw[color=black,fill=white, dashed] (1,0) ellipse (0.5 and 0.2);
\filldraw[color=black,fill=white, dashed] (3.75,0) ellipse (0.5 and 0.2);
\draw (4.25,0) arc (0:180:1.875);
\draw (3.25,0) arc (0:180:0.875);
\node at (2.5,-1) {Step 0: Original handle.};

\begin{scope}[xshift=7cm]
% Pinching
\filldraw[color=black,fill=black!5, dashed] (2.25,0) ellipse (2.5 and 0.5);
\filldraw[color=black,fill=white, dashed] (1,0) ellipse (0.5 and 0.2);
\filldraw[color=black,fill=white, dashed] (3.75,0) ellipse (0.5 and 0.2);
\draw (4.25,0) arc (0:180:1.875);
\draw (1.5,0) .. controls (2.20,2.45) and (2.30,2.45)  .. (3.25,0);
\filldraw[black] (2.3,1.85) circle (2pt);
\node at (2.5,-1) {Step 1: Pinching the  handle.};
\end{scope}

\begin{scope}[xshift=0cm,yshift=-3cm]
% Streching
\filldraw[color=black,fill=black!5, dashed] (2.25,0) ellipse (2.5 and 0.5);
\filldraw[color=black,fill=black!5, dashed] (1,0) ellipse (0.5 and 0.2);
\filldraw[color=black,fill=black!5, dashed] (3.75,0) ellipse (0.5 and 0.2);
\draw (3.75,0) arc (0:180:1.375);
\filldraw[black] (3.75,0) circle (2pt);
\filldraw[black] (1,0) circle (2pt);
\node at (2.5,-1) {Step 2: Stretching the singular point.};
\end{scope}

\begin{scope}[xshift=7cm,yshift=-3cm]
% collapsing
\filldraw[color=black,fill=black!5, dashed] (2.25,0) ellipse (2.5 and 0.5);
\filldraw[color=black,fill=black!5, dashed] (1,0) ellipse (0.5 and 0.2);
\filldraw[color=black,fill=black!5, dashed] (3.75,0) ellipse (0.5 and 0.2);
\draw (3.75,0) -- (1,0);
\node at (2.5,-1) {Step 3: Collapsing the arc.};
\end{scope}
\end{tikzpicture}
\caption{A graphical description of the pinching-stretching-collapsing procedure}
    \label{fig:pinching-stretching-collapsing}
\end{figure}

%%%%
% Acaba la figura
%%%%

Therefore 
%$$\{0,d_1,\ldots, d_r\}=A\subset D(M_1,N_1)\cap D(M_2,N_2).$$
$$\{0,d_1,\ldots, d_r\}=A\subset D(M, N_1 {\#} N_2).$$
\end{proof}

\begin{remark}
The conclusion of Lemma \ref{lem:3} does not hold if  the connected sum with the handle bodies is omitted from the source manifold. Notice that if $f\colon M\to N$ is a non-zero degree map, then $\pi(f)\big(\pi_1(M)\big)$ has finite index in $\pi_1(N)$ \cite[Tatsache (B), p.\ 86]{Hopf-pi1}. 
Therefore, maps $f_i\colon M_i\to N_i$  where $\pi_1(f_i)$ is not surjective, $i=1,2$, cannot lead to a non-zero degree map $f_1\# f_2\colon M_1\#M_2\to N_1\# N_2$ because $\pi(f_1\# f_2)\big(\pi_1(M_1\# M_2)\big)=\pi(f_1)\big(\pi_1(M_1)\big)\ast \pi(f_2)\big(\pi_1(M_2)\big)$ has not finite index in $\pi_1(N_1\# N_2)=\pi_1(N_1)\ast\pi_1(N_2)$.  To illustrate this argument,  let us consider  the manifolds $M=S^n$ and $N=\R\mathrm{P}^n$, where $n>2$. While there exist non-zero degree maps from $S^n$ to $\R\mathrm{P}^n$,  any map from $S^n\# S^n$ to $\R\mathrm{P}^n\# \R\mathrm{P}^n$ 
must be of degree zero. This arises from the fact that   $S^n\# S^n= S^n$ is simply connected while $\pi_1(\R\mathrm{P}^n\# \R\mathrm{P}^n)=\Z/2\ast\Z/2$ is infinite.
\end{remark}

Combining Lemmas \ref{lem:lemas_CL+NWW}, \ref{lem:lema_NWW}, \ref{lem:fix_3.3}, and \ref{lem:3}, we prove the following result:

\begin{proposition}\label{prop:iteratedsums}
Let $M_i, N_i, \, i=1, \ldots, r$, be oriented closed connected $n$-manifolds,
$n>2$, satisfying:
\begin{enumerate}
    \item $\pi_{n-1}(N_i)=0$ for $i=1,\ldots, r$; 
    \item $D(M_i,N_j)=\{0\}$ for $i\ne j$; 
    \item $D(M_i,N_i)\cap D(M_j,N_j)$ is finite for $i\ne j$.
\end{enumerate}
Then there exists an integer $\ell\geq 0$ such that
$$D\big((\operatornamewithlimits{\#}_{j=1}^r M_j)\# H(n,\ell), \operatornamewithlimits{\#}_{i=1}^r N_i\big)=  \bigcap_{i=1}^r D(M_i,N_i).
$$
\end{proposition}
\begin{proof}
By hypothesis, $\bigcap_{i=1}^rD(M_i,N_i)$ is a finite set containing $0$, and by Lemma \ref{lem:3}, there exists an integer $\ell\geq 0$ such that
$$\bigcap_{i=1}^rD(M_i,N_i)\subset D\big((\operatornamewithlimits{\#}_{j=1}^r M_j)\# H(n,\ell), \operatornamewithlimits{\#}_{i=1}^r N_i\big).$$     
Conversely, to show that 
$$D\big((\operatornamewithlimits{\#}_{j=1}^r M_j)\# H(n,\ell), \operatornamewithlimits{\#}_{i=1}^r N_i\big)\subset \bigcap_{i=1}^rD(M_i,N_i),$$
we first observe that
according to Lemma \ref{lem:lemas_CL+NWW},
$$D\big((\operatornamewithlimits{\#}_{j=1}^r M_j)\# H(n,\ell), N_i\big)=D(M_1,N_i)+\ldots+D(M_r,N_i)+D\big(H(n,\ell),N_i\big).$$
Thus by our hypothesis and Lemma \ref{lem:fix_3.3}, we get 
$$D\big((\operatornamewithlimits{\#}_{j=1}^r M_j)\# H(n,\ell), N_i\big)=D(M_i,N_i).$$ 
Therefore, applying Lemma \ref{lem:lema_NWW} we finally obtain: 
$$D\big((\operatornamewithlimits{\#}_{j=1}^r M_j)\# H(n,\ell), \operatornamewithlimits{\#}_{i=1}^r N_i\big)\subset \bigcap_{i=1}^r D\big((\operatornamewithlimits{\#}_{j=1}^r M_j)\# H(n,\ell),N_i\big) = \bigcap_{i=1}^rD(M_i,N_i).$$
\end{proof}

We now have all the ingredients to prove our main theorem.

%%%%%%%%%%%%%%%%%%%%%%%%%%%%%%%%%%%%%%%%%%%%%%%%%%%%%
\subsection*{Proof of Theorem \ref{thm:main-3-manifolds}}\label{sec:proofmainthm} Let $A=\{0, d_1,\ldots, d_n\}$  be a finite set of pairwise distinct  integers. We need to show that $A$ is realized by two oriented closed connected  $3$-manifolds $M, N$ in the sense that $A=D(M,N)$.

For this purpose, we consider an oriented closed connected hyperbolic surface of genus $g >1$, $\Sigma_g$. Then, for every $i \in \mathbb Z$,  let $K_i$ be the total space in the circle bundle
 $$
 S^1\to K_i \to \Sigma_g
 $$
with Euler number $e(K_i)=i$. Observe that $K_i, i \in \mathbb Z,$ is an aspherical $3$-manifold. The mapping degree set between these $3$-manifolds is fully described in {\cite[Lemma 3.4]{NWW}}:
\begin{equation}\label{eq:degree3man}
 D(K_i,K_j)=\left\{ \begin{array}{ll} \{0,j/i\}, \qquad & \text{if } i|j, \\ \{0\}, & \text{if } i{\not|} j. \end{array}\right.
\end{equation}

According to Proposition \ref{prop:clave-aritmetica}, for every positive integer $m>0$ that we fix, there exist finite sequences,  $B(i)$, $i=0,\ldots, n$, of not necessarily pairwise distinct non-zero integers,  satisfying that
\begin{equation*}
A=\bigcap_{i=0}^n  S_{B(i)}.
\end{equation*}

Now, we choose particular pairwise distinct primes $q_0,q_1,\ldots, q_n$  fulfilling the condition
 $$
 q_i>\max\{|b|\,\, |\, b\in B(i)\}, \, i =0, \ldots, n,
 $$
 and  we denote
 $$
 \alpha_i=q_i\prod\limits_{b\in B(i)}b,  \, i =0, \ldots, n.
 $$
 Then, we construct the following ``intermediate''  manifolds (that will serve us to realize each of the sums $S_{B(i)}$), for $i=0,\ldots,n$:
 \begin{align*}
M_i &= \operatornamewithlimits{\#}_{b\in B(i)} K_{\alpha_i/b} \\
N_i & = K_{\alpha_i}.
\end{align*}

Because $K_{\alpha_i}$ are aspherical $3$-manifolds, for $i=0, \ldots, n,$  we have that $\pi_2(K_{\alpha_i}) =0$, and conditions to apply Lemma  \ref{lem:lemas_CL+NWW}  hold. Therefore:
$$
D(M_i,N_j) = D(  \operatornamewithlimits{\#}_{b\in B(i)} K_{\alpha_i/b}, K_{\alpha_{j}} ) = \operatornamewithlimits{\sum}_{b\in B(i)} D(K_{\alpha_i/b}, K_{\alpha_{j}} ).
$$
Using \eqref{eq:degree3man},  we  then get that, for $i=0, \ldots, n,$
\begin{align*}
D(M_i,N_i)& =S_{B(i)}\, ,\,\text{and}\\
D(M_i,N_j)&=\{0\},\,\text{for $i\neq j$.}
\end{align*}

Since all the conditions to apply Proposition \ref{prop:iteratedsums} plainly hold for the manifolds $M_i$ and $N_i$, there exists an integer $\ell\geq 0$ such that, for
 \begin{align*}
 {M}&= M_0\# M_1\# \ldots \# M_n\# H(3,\ell), \\
  N&= N_0\# N_1\# \ldots \# N_n,
  \end{align*}
we have that
$$
 D(M,N)=\bigcap_{i=0}^n  S_{B(i)} = A,
$$
and the proof of Theorem \ref{thm:main-3-manifolds} is complete.

\hfill{$\square$}

\begin{remark}\label{rem:MNinflexible}
We claim  that the $3$-manifolds $K_i$, $M_i$, $N_i$, $\operatornamewithlimits{\#}_{i=0}^n{M_i}$, and $N$ involved in the previous proof are inflexible (see also Remark \ref{rem:inflexible}). It is clear, by \eqref{eq:degree3man}, that  $K_i, \, i \in \mathbb Z$, are inflexible.  Now, proceeding along the lines of the proof of Theorem  \ref{thm:main-3-manifolds}, we apply repeatedly Lemma \ref{lem:lema_NWW} and Lemma \ref{lem:lemas_CL+NWW} to get the inflexibility property. On the one hand, we obtain that $D(M_i, M_j) = \{0\}$ for $i \neq j$, and on the other hand
$$
D(\operatornamewithlimits{\#}_{i=0}^n{M_i}, \operatornamewithlimits{\#}_{i=0}^n{M_i}) \subset \bigcap_{i=0}^n  D( M_i, M_i).
$$
 Also, by Lemma \ref{lem:lema_NWW},
$$
D(M_i, M_i) = D(M_i, \operatornamewithlimits{\#}_{b\in B(i)} K_{\alpha_i/b}) \subset \bigcap_{b\in B(i)} D (M_i,  K_{\alpha_i/b})
$$
and using Lemma \ref{lem:lemas_CL+NWW},
$$
D(M_i, K_{\alpha_i/b}) = D(  \operatornamewithlimits{\#}_{b'\in B(i)} K_{\alpha_i/b'}, K_{\alpha_i/b} ) = \operatornamewithlimits{\sum}_{b'\in B(i)} D(K_{\alpha_i/b'}, K_{\alpha_i/b} ).
$$
Now, by Equation \eqref{eq:degree3man}, $D(K_{\alpha_i/b'}, K_{\alpha_i/b} )$ is either $\{0\}$ or $\{0, b'/b\}$ whenever $b | b'$. Hence, $D(M_i, K_{\alpha_i/b}) $ is bounded, and  so   $D(M_i, M_i) $ and $D(\operatornamewithlimits{\#}_{i=0}^n{M_i},\operatornamewithlimits{\#}_{i=0}^n{M_i})$ are bounded. Hence $M_i$, $i=1, \ldots, n$ and  $\operatornamewithlimits{\#}_{i=0}^n{M_i}$ are  inflexible manifolds.  The same arguments work for $N_i$ and $N$, and we conclude.
\end{remark}

\begin{remark}\label{rem:notunique}
The $3$-manifolds involved in Theorem \ref{thm:main-3-manifolds} are not unique, even from a homotopical point of view. Consider any $3$-manifold $P$ with a finite fundamental group $\pi_1(P)$ (for example the Poincar\'e sphere $S^3/\mathrm{SL}(2,5)$, or any lens space $L(p,q)$). Let $K_i, K_j$ be the building blocks in the proof of Theorem \ref{thm:main-3-manifolds}. According to  Lemma  \ref{lem:lemas_CL+NWW}, we have:
 $$
 D(P\#K_i,K_j)=D(P,K_j)+D(K_i,K_j)=\{0\}+D(K_i,K_j)=D(K_i,K_j),
 $$
since $K_j$ is aspherical, and $\pi_1(K_j)$ torsion free while $\pi_1(P)$ is finite. Therefore, in the construction of $M$,  we can replace any $K_i$  by $P\#K_i$ to yield a new $3$-manifold $M'$. This new manifold has torsion in its fundamental group $\pi_1(M')$ unlike $\pi_1(M)$ which was torsion free, and verifies $D(M,N)=D(M',N)$.
\end{remark}

%%%%%%%%%%%%%%%%%%%%%%%%%%%%%%%%%%%%%%%%%%%%%%%%%%%%%
%%%%%%%%%%%%%%%%%%%%%%%%%%%%%%%%%%%%%%%%%%%%%%%%%%%%%
\section{Spherical fibrations over inflexible Sullivan models: \\rational mapping degree set}\label{sec:rational-mappingdegree}
%%%%%%%%%%%%%%%%%%%%%%%%%%%%%%%%%%%%%%%%%%%%%%%%%%%%%
%%%%%%%%%%%%%%%%%%%%%%%%%%%%%%%%%%%%%%%%%%%%%%%%%%%%%

In this section we prove Theorem \ref{thm:main-rational}, which can be thought of as the rational version of Theorem \ref{thm:main-3-manifolds}. Rational homotopy theory provides an equivalence of categories between the category of simply connected  rational spaces and the category of certain differential graded algebras, the so-called Sullivan minimal models.  We refer to \cite{FHT} for basics facts in rational homotopy theory.

More concretely, if $V$ is a graded rational vector space, we write $\Lambda V$ for the free commutative graded algebra on $V$.  A Sullivan algebra $(\Lambda V, \partial)$ is a commutative differential graded algebra (cdga for short) which is free as commutative graded algebra on a simply connected graded vector space $V$ of finite dimension in each degree, and such that $V$ admits a basis  $x_\alpha$ indexed by a well-ordered set with $\partial (x_\alpha)  \in \Lambda (x_\beta)_{\beta <\alpha}$. It is minimal if in addition $\partial(V) \subset \Lambda^{\geq 2}V$. A Sullivan minimal model of a cdga \((A, d)\) is a Sullivan minimal cdga \((\Lambda V, d)\) that is quasi-isomorphic to \((A, d)\).

Recall that, to each space \(X\), Sullivan associated a cdga of forms with rational coefficients,  $A_{PL}(X)$, whose cohomology is isomorphic to the cohomology of $X$ with rational coefficients:
\[
H^*(A_{PL}(X)) \cong H^*(X; \mathbb{Q})
\]
The  Sullivan minimal model of  the cdga $A_{PL}(X)$ is called the minimal model of \(X\) that, in this paper, we denote by $A_X$.  In the case of oriented closed simply connected manifolds $M$, the cohomology of the associated minimal model $A_M$ is a Poincar\'e duality algebra. In particular $A_M$
 has a cohomological fundamental class $[A_M]\in H^{\ast} (A_M) \cong H^{\ast}(M; \mathbb Q)$ which corresponds to the rational cohomological fundamental class $[M]_\Q$ of $M$.

\emph{Ellipticity} for a Sullivan minimal model $(\Lambda V, \partial)$ means that both $V$  and $H^\ast (\Lambda V)$ are finite-dimensional.
Hence, the cohomology is  a Poincar\'e duality algebra \cite{Hal} and one can easily compute the degree of its fundamental cohomological  class \cite[Theorem 32.6]{FHT}.
In particular one can introduce the notion of mapping degree between elliptic Sullivan minimal models and also translate the notion of inflexibility.

Let $(\Lambda V, \partial)$ be an elliptic Sullivan minimal model. Let $\mu \in (\Lambda V)^n$ be a representative of its cohomological fundamental class. Then $(\Lambda V, \partial)$
is \emph{inflexible} if for every cdga-morphism
 $$
 \varphi\colon(\Lambda V, \partial) \to (\Lambda V,\partial),
 $$
we have $\deg(\varphi)=0,\pm 1$, where $H^*(\varphi)([\mu])=\deg(\varphi) [\mu]$.

%%%%%%%%%%%%%%%%%%%%%%%%%%%%%%%%%%%%%%%%%%%%%%%%%%%%%
\subsection{Rational mapping degree set and connected sums} The following results establish, under certain restrictions, the relationship  between  rational mapping degree sets and  connected sums of manifolds.
\begin{lemma}\label{prop_2:lemas_CL+NWW}
Let $M_i$, $i=1,\ldots, r,$ and $N$ be oriented closed connected $n$-manifolds with  $\pi_{n-1}(N_{(0)})=0$. Then
$$
D_\Q(\operatornamewithlimits{\#}_{i=1}^r M_i,N)= \sum_{i=1}^r D_\Q(M_i,N).
$$
\end{lemma}

\begin{proof}
By \cite[Lemma II.2]{CL}, since $\pi_{n-1}(N_{(0)})=0$,  the following  holds:
$$
D(M_1\# M_2,N) \subset D_\Q(M_1,N)+D_\Q(M_2,N)
.$$
However,  a stronger result is demonstrated in the proof of \cite[Lemma II.2]{CL}. Namely,
$$
D_\Q(M_1\# M_2,N) \subset D_\Q(M_1,N)+D_\Q(M_2,N).
$$
A straightforward inductive argument shows that 
$$
D_\Q(\operatornamewithlimits{\#}_{i=1}^r M_i,N) \subset \sum_{i=1}^r D_\Q(M_i,N),
$$
hence, it suffices to prove the other inclusion. To that end, one can apply the same arguments as in \cite[Lemma 7.8]{CMuV1}:
let $q_{(0)}\colon (\#_{i=1}^r M_i)_{(0)}\to \vee_{i=1}^r (M_i)_{(0)}$ denote the rationalization of the pinching map. Then for every given maps $f_i\colon (M_i)_{(0)}\to N_{(0)}$, the composition
$$
(\vee_{i=1}^r f_i)\circ q_{(0)}\colon  (\operatornamewithlimits{\#}_{i=1}^r M_i)_{(0)}\to N_{(0)}
$$
has degree $\sum_{i=1}^r\deg(f_i)$ and the result follows.
\end{proof}

We give a precise definition of connected sum in the world of cdgas:
\begin{definition}\label{def:connected_sum_dga}
Let ${A}_i$, $i=1,2$, be connected cdgas and let $a_i\in A_i$, $i=1,2$,  be elements of the same degree. The connected sum of the pairs $(A_i, [a_i])$, $i=1,2$, is the cdga
 $$
 (A_1, [a_1])\#(A_2, [a_2]) \buildrel\text{\scriptsize def}\over{:=} (A_1\oplus_{\mathbb{Q}}A_2 )/I\, ,
 $$
where $A_1\oplus_{\mathbb{Q}}A_2\buildrel\text{\scriptsize def}\over{:=}(A_1 \oplus A_2)/\mathbb{Q}\{(1,-1)\}$, and
$I\subset A_1\oplus_{\mathbb{Q}}A_2$ is the differential ideal generated by $a_1-a_2$.
\end{definition}

Connected sums of cdgas provide rational models for connected sums of oriented manifolds. Indeed,  for  $M_i, \, i=1, 2$  oriented closed simply connected $n$-manifolds, with Sullivan minimal model $A_{M_i}$,  let $m_i$ be a representative of the cohomological fundamental class  of $A_{M_i}$, for $i=1,2$.  By \cite[Theorem 7.12]{CMuV1}
\begin{equation}\label{eq:modelsum}
(A_{{M}_1}, [m_1])\# ({A_{M}}_2, [m_2])
\end{equation}
is a rational model of $M_1\# M_2$.

We use \eqref{eq:modelsum} above to prove a rational version of Proposition \ref{prop:iteratedsums} that does not involve handle-bodies:
\begin{proposition}\label{prop:3_rational}
Let $M_i, N_i, \, i=1, \ldots, r$, be oriented closed connected $n$-manifolds, such that $\pi_{n-1}(N_j)\otimes\Q=0$, $j=1,\ldots, r$, and $D_\Q(M_i,N_j)=\{0\}$,  $i\neq j$.
Then
 $$
  D_\Q(\operatornamewithlimits{\#}_{j=1}^r M_j, \operatornamewithlimits{\#}_{i=1}^rN_i)= \bigcap_{i=1}^r D_\Q(M_i,N_i).
 $$
\end{proposition}
\begin{proof}
According to Lemma \ref{prop_2:lemas_CL+NWW} and the rational version of Lemma \ref{lem:lema_NWW} (which can be proved following the same arguments as in \cite[Lemma 4.3]{NWW}),   we get that  $$ D_\Q(\operatornamewithlimits{\#}_{j=1}^r M_j, 
\operatornamewithlimits{\#}_{i=1}^rN_i)\subset \bigcap_{i=1}^r D_\Q(M_i,N_i). $$

Conversely,  let  $(A_{M_i}, [m_i] )$ and $(A_{N_i}, [n_i] )$ be Sullivan minimal models of $(M_i, [M_i])$ and $(N_i, [N_i])$ respectively, $i=1,\ldots, r$. For
 $$
 0\ne d\in \bigcap_{i=1}^r D_\Q(M_i,N_i)
 $$ 
 there exists $f_i\colon  A_{N_i}\to A_{M_i}$ with $f_i(n_i)=d\cdot m_i+\alpha_i$ and $\alpha_i$ a coboundary, $i=1,\ldots, r$.
Because $\#_{i=1}^r(A_{M_i}, [m_i+\alpha_i/d])$ and $\#_{i=1}^r(A_{N_i}, [n_i])$ are rational models for $\#_{i=1}^r M_i$ and $\#_{i=1}^rN_i$ respectively, then the morphisms $f_i$ give rise to a well-defined cdga-morphism
$$
\operatornamewithlimits{\#}_{i=1}^r f_i\colon \operatornamewithlimits{\#}_{i=1}^r(A_{N_i}, [n_i])\to \operatornamewithlimits{\#}_{i=1}^r(A_{M_i}, [m_i+\alpha_i/d])
$$
defined by $(\#_{i=1}^r f_i)(x)= f_i(x)$ if  $x\in A_{N_i},$
and whose degree is $\deg(\#_{i=1}^r f_i) = d$.
\end{proof}

%%%%%%%%%%%%%%%%%%%%%%%%%%%%%%%%%%%%%%%%%%%%%%%%%%%%%
\subsection{Inflexible Sullivan minimal models of inflexible manifolds} Following the same strategy as in Section \ref{sec:degree_connected}, we consider spherical fibrations  over certain elliptic and inflexible Sullivan minimal models (Definition \ref{def:K_i-rational}), whose total spaces are Sullivan minimal models of  inflexible manifolds (see Lemma \ref{lem:K_d-elliptic-inflexible}). These manifolds will be the building blocks to construct, by means of iterated connected sums,  manifolds that realize finite sets of rational numbers.

\begin{definition}\label{def:K_i-rational}
Let $(A,\partial)$ be an elliptic, inflexible Sullivan minimal model of formal dimension $2m, \, m \geq 1,$ such that $\pi_j (A)=0$ for $j\geq 2m-1$. Fix $\mu\in A$ a representative of its cohomological fundamental class. Then, for every non-zero $q\in\Q$,  we define the following Sullivan minimal model $$(K_q(A),\partial):= (A\otimes\Lambda(y_{2m-1}),\partial)$$ that extends the differential  of $A$ by $\partial(y_{2m-1})=q \mu.$
\end{definition}

\begin{remark}\label{rem:K_d-is-fibration}
Notice that $({K}_q({A}),\partial)$ is the total space in the rational $S^{2m-1}$-fiber sequence:
\begin{equation*}
(\Lambda(y_{2m-1}),0)\longleftarrow (K_q(A),\partial) \longleftarrow (A,\partial),
\end{equation*}
whose Euler class is $q[\mu]$.
\end{remark}

\begin{lemma}\label{lem:K_d-elliptic-inflexible}
Let $(A,\partial)$ be an elliptic,  inflexible Sullivan minimal model of formal dimension $2m, \, m \geq 1$, such that $\pi_j(A)=0$ for $j\geq 2m-1$. Fix $\mu\in A$ a representative of the fundamental class of $A$, and let $x \in A$ such that $\partial(x)=\mu^2$. Then, for every non-zero $q\in\Q$,  $\big(K_q( A),[y_{2m-1}\mu-qx]\big)$ is the  Sullivan minimal model of an oriented closed connected inflexible $(4m-1)$-manifold $M_{{K}_q}$, with the same connectivity as $( A, \partial)$.
\end{lemma}

\begin{proof}
According to \cite[Proposition 3.1]{CV2}, $(K_q( A),\partial)$ is an elliptic Sullivan model of formal dimension $4m-1$ where $y_{2m-1}\mu-qx$ is a representative of its cohomological fundamental class. By \cite[Lemma 3.2]{CV2},  $(K_q ( A), \partial)$ is an inflexible algebra because $( A,\partial)$ is so.  Now, since its formal dimension is $4m-1\equiv 3\mod 4$,  the obstruction theory of Sullivan \cite[Theorem (13.2)]{Sullivan} and
Barge \cite[Th{\'e}or{\`e}me 1]{Barge} guarantees  that $(K_q( A),[y_{2m-1}\mu-qx])$ is the Sullivan minimal model of an oriented closed
simply connected manifold $M_{{K}_q}$. Finally, by \cite[Proposition A.1]{CMV1},
  $M_{K_q}$ and $( A, \partial)$  have the same connectivity.\end{proof}

We compute the rational mapping degree set between the manifolds appearing in the previous lemma:
\begin{lemma}\label{lem:degrees-K_d}
Let $p, q\in \mathbb Q$ be non-zero, and  $M_{K_p}$ and $ M_{K_q}$  be the oriented closed simply connected manifolds from Lemma \ref{lem:K_d-elliptic-inflexible} whose Sullivan minimal models are, respectively,  $({K}_p( A), \partial)$  and $({K}_q( A), \partial)$  from Definition \ref{def:K_i-rational}. Then
$$
D_\Q(M_{K_p}, M_{K_q})=\{0, q/p\}.
$$
\end{lemma}

\begin{proof}

We follow the ideas in \cite[Lemma 3.2]{CV2}. Let $f\colon (K_q( A),\partial) \to (K_p ( A),\partial)$ be a morphism of non-zero degree $d\in\Q$, that is,
\begin{equation}\label{eq:degree-d}
f(y_{2m-1}\mu-qx)=d(y_{2m-1}\mu-px)+\alpha
\end{equation}
where $\alpha$ is a coboundary. Here $x,y$ on the left-hand side belong to the algebra $K_q(A)$, while $x,y$ on the right-hand side belong to the algebra  $K_p(A)$. Now, as all generators of $A$ have
degrees strictly less than $\deg(y)=2m-1$, $f$ induces a 
morphism $f|_{ A}\colon ( A,\partial)\to ( A,\partial)$, which
is of non-zero degree. On the one hand $f(\mu)=\widetilde{d}\mu+\beta_1$ and $f(x)=\widetilde{d}\,{}^2x+\beta_2$ where $\beta_1,\beta_2$ are coboundaries, and  $\widetilde{d}\in\{-1,1\}$. On
the other hand,  $f(y_{2m-1})=ay_{2m-1}+\gamma$ where $a\in\Q$ and $\gamma$ is a coboundary.

Because $f(\partial y_{2m-1})=\partial f(y_{2m-1})$, we get that $ap=q\widetilde{d}$ and $\beta_1=0$.  Hence $a=\widetilde{d}\,(q/p)$ and
\begin{align*}
    f(y_{2m-1}\mu-qx)  &=(\widetilde{d}\, q/p \,y_{2m-1}+\gamma\big)(\widetilde{d}\mu)-q(\widetilde{d}\,^2x+\beta_2)\\
                &= (\widetilde{d}\,^2 {q/p})(y_{2m-1}\mu -px)-q\beta_2\\
                &= (q/p)(y_{2m-1}\mu -px)-q\beta_2\qquad\text{(recall $\widetilde{d}\in\{-1,1\}$)}.
\end{align*}
By comparing this equation to \eqref{eq:degree-d}, we obtain that $d=q/p$ and the proof is complete.
\end{proof}

We  illustrate the existence of elliptic,  inflexible Sullivan minimal models satisfying the conditions from Definition  \ref{def:K_i-rational} and Lemma  \ref{lem:K_d-elliptic-inflexible}:

\if{}
\begin{definition}\label{def:K_i-grafos-rational}
Let $\G$ be a connected finite simple graph with more that one vertex, \emph{i.e.,} $|V (\G)|>1$. Given an integer $k\geq 1$, let  $({A}_k(\G),\partial)$ be the $(30k+17)$-connected elliptic and inflexible Sullivan  algebra constructed in \cite[Definition 2.1]{CMV1}, whose formal dimension is $2m=540k^2 + 984k + 396 + |V (\G)|(360k^2 + 436k + 132)$ and $\pi_j \big({A}_k(\G) \big)=0$ for $j\geq 2m-1$. Fix $\mu\in{A}_k(\G)$ a representative of the cohomological fundamental class. Then for every non-zero $q\in\Q$, define the following Sullivan minimal model
 $$(
 K_q(\G, k),\partial):=({A}_k(\G)\otimes\Lambda(y_{2m-1}),\partial)
 $$  that extends the differential of ${A}_k(\G)$ by $\partial(y_{2m-1})=q \mu.$
\end{definition}

\fi

\begin{definition}\label{def:K_i-grafos-rational}
Let $\G$ be a connected finite simple graph with more than one vertex, \emph{i.e.,} $|V (\G)|>1$. Given an integer $k\geq 1$, let  $({A}_k(\G),\partial)$ be the $(30k+17)$-connected elliptic and inflexible Sullivan  algebra constructed in \cite[Definition 2.1]{CMV1}, whose formal dimension is $2m=540k^2 + 984k + 396 + |V (\G)|(360k^2 + 436k + 132)$ and $\pi_j \big({A}_k(\G) \big)=0$ for $j\geq 2m-1$. Fix $\mu\in{A}_k(\G)$ a representative of the cohomological fundamental class. Then,  for every non-zero $q\in\Q$, we denote by $(K_q(\G, k),\partial)$ the  Sullivan minimal model associated to $({A}_k(\G),\partial)$ introduced in Definition \ref{def:K_i-rational}, that is
 $$(
 K_q(\G, k),\partial):= (K_q( {A}_k(\G)),\partial)
 .$$
\end{definition}

\begin{remark}
\label{rmk:K_d-grafos-elliptic-inflexible}
Because the conditions from Lemma \ref{lem:K_d-elliptic-inflexible} hold,  $(K_q(\G, k), \partial)$ is a Sullivan model of an oriented closed  $(30k+17)$-connected
 inflexible  $(4m-1)$-manifold $M_{K_q(\G, k)}$, where  $2m=540k^2 + 984k + 396 + |V (\G)|(360k^2 + 436k + 132)$.
\end{remark}

\begin{lemma}\label{lem:degrees-K_d-grafos-rational}
Let $\G_1$ and $\G_2$ be connected finite simple graphs with  $|V (\G_1)|=|V (\G_2)|>1$. Given a positive integer $k\geq 1$,  and a non-zero $p_i \in \mathbb Q, \, i=1,2$,  consider a manifold $M_{K_{p_i}(\G_i, k)},  \, i=1,2$ as in Remark \ref{rmk:K_d-grafos-elliptic-inflexible}. Then $$D_\Q(M_{K_{p_1}(\G_1, k)}, M_{K_{p_2}(\G_2, k)})=
\begin{cases}
\{0, p_2/p_1\}, &\text{ if \, $\G_1\cong \G_2$,}\\
\{0\}, &\text{ otherwise.}
\end{cases}$$
\end{lemma}
\begin{proof}
Let $(K_{p_i}(\G_i, k), \partial) = ({A}_k(\G_i)\otimes\Lambda({y_i}),\partial )$, introduced in Definition \ref{def:K_i-grafos-rational},  be the Sullivan model of the manifold $M_{K_{p_i}(\G_i, k)},$  where $ \partial (y_i ) = p_i \mu_i$  for $ \mu_i $ a representative of the cohomological fundamental class of  $A_k(\G_i)$, $i=1,2$.  Recall from Lemma \ref{lem:K_d-elliptic-inflexible} that for $x_i\in A_k (\G_i)$ satisfying $\partial(x_i)=\mu_i^2$, the element $y_i \mu_i-p_ix_i$ is a representative of the cohomological fundamental class of  $(K_{p_i}(\G_i, k), \partial)$, $i=1,2$.

With these constructions in mind, we follow the ideas from \cite[Lemma 3.2]{CV2}. Consider a morphism of non-trivial degree  $d\in\Q$:
$$
f\colon (K_{p_2}(\G_2, k),\partial)\to (K_{p_1}(\G_1, k), \partial).
$$
Then $f(y_{2}\mu_2-p_{2}x_2)=d(y_{1}\mu_1-p_{1}x_1)+\alpha$ with $\alpha$ a coboundary. By a degree argument, the morphism $f$ induces a non-trivial degree morphism $$f|_{A_k(\G_2)}\colon (A_k(\G_2),\partial)\to ({A}_k(\G_1),\partial).$$  Focusing specifically on $f|_{A_k(\G_2)}$, the arguments in \cite[Lemma 2.12]{CMV1} (see also \cite[Remark 2.8]{CV2}),  show that it is induced by a graph full monomorphism $\sigma\colon \G_1\to \G_2$. Now,  since $|V (\G_1)|=|V (\G_2)|$,  $\sigma$ is indeed an isomorphism of graphs,  and $f(\mu_2)=\mu_1+\beta_1$ and $f(x_2)=x_1+\beta_2$ with  $\beta_1,\beta_2$ coboundaries, by \cite[Lemma 2.12]{CMV1}.

Finally, by another degree reasoning argument, one obtains that $f(y_2)=ay_1+\gamma$ where $a$ is a non-zero rational number,  and $\gamma$ is a coboundary. We conclude as in  the proof of Lemma \ref{lem:degrees-K_d}.
\end{proof}

%%%%%%%%%%%%%%%%%%%%%%%%%%%%%%%%%%%%%%%%%%%%%%%%%%%%%
\subsection{Proof of Theorem  \ref{thm:main-rational}}\label{sub:theoB}

Let $A=\{0, d_1,\ldots, d_n\}$ where $d_1,d_2,\ldots,d_n$ are pairwise different non-zero rational numbers.  Fix an integer $k\geq 1$.  According to Corollary \ref{cor:clave-aritmetica-racional}, there exist finite sequences of not necessarily pairwise distinct non-zero rational numbers $B(i)$, $i=0,\ldots, n$, such that
\begin{equation*}
A=\bigcap_{i=0}^n  S_{B(i)}.
\end{equation*}

Choose $\G_0,\G_1,\ldots, \G_n,$ pairwise non-isomorphic connected finite simple graphs, such that $|V (\G_i)|=|V (\G_j)|>1$ for every $i,j=0,\ldots, n$.
According to Remark \ref{rmk:K_d-grafos-elliptic-inflexible}, we define the $(30k+17)$-connected manifolds
\begin{align*}
M_i &= \operatornamewithlimits{\#}_{b\in B(i)}  M_{K_{b^{-1}}(\G_i, k)} \\
N_i & = M_{K_{1}(\G_i, k)},
\end{align*}
for $i=0,\ldots,n$. By Lemmas \ref{lem:degrees-K_d-grafos-rational} and \ref{prop_2:lemas_CL+NWW},
%the rational version of Lemma \ref{lem:lema_NWW} (which can be proved following the same arguments as in \cite[Lemma 4.3]{NWW}), 
we have that
\begin{gather*}
D_\Q(M_i,N_i) =S_{B(i)}\, ,\,\text{and}\\
D_\Q(M_i,N_j)=\{0\},\,\text{for $i\neq j$.}
\end{gather*}

Finally, define
  \begin{align*}
  M&= M_0\# M_1\# \ldots \# M_n, \\
  N&= N_0\# N_1\# \ldots \# N_n .
  \end{align*}
and use Proposition \ref{prop:3_rational} to get
 $$
 D_\Q(M,N)=\bigcap_{i=0}^N  S_{B(i)} = A.
$$

%%%%%%%%%%%%%%%%%%%%%%%%%%%%%%%%%%%%%%%%%%%%%%%%%%%%%
%%%%%%%%%%%%%%%%%%%%%%%%%%%%%%%%%%%%%%%%%%%%%%%%%%%%%
\section{From unstable Adams operations to mapping degree sets}\label{sec:adamsoperations}
%%%%%%%%%%%%%%%%%%%%%%%%%%%%%%%%%%%%%%%%%%%%%%%%%%%%%
%%%%%%%%%%%%%%%%%%%%%%%%%%%%%%%%%%%%%%%%%%%%%%%%%%%%%

We recall the basics on unstable Adams operations following Jackowski-McClure-Oliver's work \cite{JMO1, JMO2, JMO-revisited}.
Given a compact connected Lie group $G$, a self-map $f\colon BG \to BG$
is  called  an  unstable  Adams  operation  of degree  $r\geq 0$, if  $H^{2i}(f;\Q)$ is the multiplication  by  $r^i$ for  each $i>0$ \cite[p.\ 183]{JMO1}. Equivalently,  if $T\subset G$ is a maximal torus, then $f\colon BG \to BG$
is  an  unstable  Adams  operation  of degree  $r\geq 0$, if there exists a homotopy commutative diagram
\begin{equation}\label{eq:Adams_vs_admissible}
 \begin{tikzcd}
BT \arrow[r, "B\psi_r"] \arrow[d] & BT \arrow[d]\\
BG \arrow[r, "f"]  & BG
\end{tikzcd}
\end{equation}
where $\psi_r\colon T \to T$ is defined by $\psi_r(t)=t^r$ \cite[Section 2]{JMO-revisited}.

For a given simple Lie group $G$ with Weyl group $W_G$, an unstable Adams operation  of degree $r>0$ exists if and only if $\gcd (r, |W_G|) = 1$, and moreover, this operation is unique up to homotopy \cite[Theorem 2]{JMO1}, \cite[Theorem 2.1]{JMO-revisited}. In particular, when  $G=\mathrm{SO}(2m-1)$ or $G= \mathrm{SO}(2m)$,  $m > 1$, unstable Adams operations of degree $r>0$ exist if $\gcd(r, m!) =1$.  In what follows, we denote by $\varphi^r$ the unstable Adams operation of degree $r>0$ on $B\mathrm{SO}(2m-1)$ and $B\mathrm{SO}(2m)$. Notice that since they are unique, then $\varphi^s\circ\varphi^r \simeq \varphi^{sr}$. Moreover, as the standard maximal torus of $\mathrm{SO}(2m-1)$ extends to a maximal torus of $\mathrm{SO}(2m)$ while $|W_{\mathrm{SO}(2m-1)}|$ divides $|W_{\mathrm{SO}(2m)}|$, then for $r>0$ such that $\gcd(r, m!) =1$, there exists a homotopy commutative diagram
\begin{equation}\label{eq:Adams_inclusions}
 \begin{tikzcd}
\mathrm{SO}(2m-1) \arrow[r, "\varphi^r"] \arrow[d] & B\mathrm{SO}(2m-1) \arrow[d]\\
B\mathrm{SO}(2m) \arrow[r, "\varphi^r"]  & B\mathrm{SO}(2m),
\end{tikzcd}
\end{equation}
where vertical maps are induced by the inclusion $\mathrm{SO}(2m-1)\subset \mathrm{SO}(2m).$

Henceforward, $(\Sigma, [\Sigma])$ is a fixed oriented closed connected $2m$-manifold  whose rationalization $(\Sigma_{(0)}, [\Sigma]_{\mathbb Q})$ is inflexible and $\pi_j(\Sigma_{(0)})=0$ for $j\geq 2m-1$. Let $({A}_\Sigma,\partial)$ be a Sullivan minimal model of $\Sigma$. In what follows $\pi\colon \Sigma\to S^{2m}$ denotes a fixed map of degree $1$, which always exists (e.g.\ \cite[Exercise 7, p.\ 258]{Hatcher}).

%Denote by $\pi\colon \Sigma\to S^{2m}$ the map obtained by collapsing the $(2m-1)$-skeleton of $\Sigma$.

\begin{lemma}\label{lem:existe_iota}
Let $X_{2m}\in H^{2m}(B\mathrm{SO}(2m);\Z)$ be the Euler class of the spherical fiber sequence
$$
S^{2m-1}\to B\mathrm{SO}(2m-1)\to B\mathrm{SO}(2m),
$$
thus $X_{2m}$ is a non-torsion integral cohomology class \cite[Theorem 1.5, Equation (2.1)]{Brown}.  
Then, there exists $\iota\colon S^{2m}\to  B\mathrm{SO}(2m)$, a non-torsion element in $\pi_{2m}(B\mathrm{SO}(2m))$, and a non-zero integer $\kappa\in\Z$ such that $H^*(\iota;\Z)(X_{2m})=\kappa [S^{2m}]$.
\end{lemma}

\begin{proof}
Recall that $\pi_{2m}(B\mathrm{SO}(2m))\cong \pi_{2m-1}(\mathrm{SO}(2m))$. By \cite[Corollary IV.6.14]{Mimura-Toda} (see also \cite[p.\ 161]{Kervaire}), $\pi_{2m-1}(\mathrm{SO}(2m))$  contains a copy of $\Z$ inducing, for every prime $p$, the $p$-local (thus rational) splitting  $\mathrm{SO}(2m)\simeq_{(p)} \mathrm{SO}(2m-1)\times S^{2m-1}$ \cite[Corollary IV.6.21]{Mimura-Toda}. Let $\iota$ be a generator of such a copy of $\Z$ in $\pi_{2m}(B\mathrm{SO}(2m))$.

By construction, $H^*(\iota;\Q)$ is non-trivial on the Euler class of the rational fiber sequence $$S^{2m-1}_{(0)}\to B\mathrm{SO}(2m-1)_{(0)}\to B\mathrm{SO}(2m)_{(0)},$$ which is just $X_{2m}\otimes_\Q 1$. Therefore,  $H^*(\iota;\Z)(X_{2m})=\kappa [S^{2m}]$ for some  non-zero $\kappa\in\Z$.
\end{proof}

From this point forward we fix both $\iota\colon S^{2m}\to  B\mathrm{SO}(2m)$ and the non-zero integer $\kappa$, as defined by Lemma \ref{lem:existe_iota}.

\begin{definition}\label{def:E_r}
Given integers $r>0, m > 1,$ with $r$ coprime to $m!$, we define:
\begin{enumerate}
\item The  oriented $(4m-1)$-manifold  $E_{r^m}$ as  the total space in the principal spherical $\mathrm{SO}(2m)$-fiber bundle
\begin{equation}\label{eq:pullback_def_K_k}
 \begin{tikzcd}
S^{2m-1} \arrow[r, equal] \arrow[d] & S^{2m-1} \arrow[d]\\
E_{r^m} \arrow[r] \arrow[d] \arrow[dr, phantom, "\scalebox{1.5}{$\lrcorner$}" , very near start] & B\mathrm{SO}(2m-1) \arrow[d]\\
\Sigma \arrow[r, "\phi_r"] & B\mathrm{SO}(2m),
\end{tikzcd}
\end{equation}
where $\phi_r=\varphi^r\circ\iota\circ\pi$.
\item The oriented  $(4m-1)$--manifold $E_{-r^m}$ obtained by reversing the original orientation on the manifold $E_{r^m}$  introduced above.
\end{enumerate}
\end{definition}

\begin{remark}
By construction, the Euler class of the spherical fiber bundle over $\Sigma$ given in diagram \eqref{eq:pullback_def_K_k} is $\kappa r^m[\Sigma]$.
\end{remark}

Recall from the beginning of this section that $(\Sigma, [\Sigma])$ is a fixed oriented closed connected  $2m$-manifold where $({A}_\Sigma,\partial)$ is its Sullivan minimal model.

\begin{lemma}\label{lem:rational_model_E_k}
Let $E_{r^m}$  be the manifold introduced in Definition \ref{def:E_r}.  A Sullivan minimal model of $E_{r^m}$ is $K_{\kappa r^m}(A_\Sigma)$ as given in Definition \ref{def:K_i-rational}. Therefore $E_{r^m}$ is rationally equivalent to  $M_{{K}_{\kappa r^m}}$,  the manifold given in Lemma \ref{lem:K_d-elliptic-inflexible}.
\end{lemma}

\begin{proof}
As it was pointed out in Remark \ref{rem:K_d-is-fibration}, ${K}_{\kappa r^m}(A_\Sigma)$ is a Sullivan minimal model for the total space in a rational $S^{2m-1}$-fiber sequence whose Euler class is $\kappa r^m[\Sigma]_\Q$. It coincides with the Euler class of the rationalization of the spherical $\mathrm{SO}(2m)$-fiber bundle in diagram~\eqref{eq:pullback_def_K_k}.  Therefore ${K}_{\kappa r^m}(A_\Sigma)$ is a Sullivan minimal model for $E_{r^m}$.
\end{proof}

\begin{lemma}\label{lem:grados_E_i}
Let $i,j, m$ be positive integers, $m>1$, such that $\gcd(i,m!)=\gcd(j,m!)=1$, and let $E_{r^m}$, $r=i,j$, be the $(4m-1)$-manifold introduced in Definition \ref{def:E_r}. Then
 $$
D(E_{i^m},E_{j^m})=
\begin{cases}
\{0, (j/i)^m\}, &\text{ if $i|j$,}\\
\{0\}, &\text{ if $i{\not|} j$.}
\end{cases}
$$
\end{lemma}

\begin{proof}
By  Lemma \ref{lem:rational_model_E_k},  the manifolds $E_{r^m}$  and $M_{K_{\kappa r^m}}$ are rationally equivalent, for every $0<r\in\Z$. Therefore:
\begin{align*}
D(E_{i^m},E_{j^m}) \subset D_\Q(E_{i^m},E_{j^m})\cap\Z  & =  D_\Q (M_{K_{\kappa i^m}}, M_{K_{\kappa j^m}}) \cap \Z \\
&= \{0, (j/i)^m\}\cap \Z\,\text{ (by Lemma \ref{lem:degrees-K_d})}\\
&=\begin{cases}
\{0, (j/i)^m\}, &\text{ if $i|j$,}\\
\{0\}, &\text{ if $i{\not|} j$.}
\end{cases}
\end{align*}

The proof will be completed if we construct a map $f\colon E_{i^m}\to E_{j^m}$ of degree $(j/i)^m$ when $i|j.$
To this end, let us  suppose that  $j=di$, $d\in\Z$, and recall that unstable Adams operations  satisfy that  $\varphi^j \simeq \varphi^d\circ \varphi^i$. Therefore, by construction (see Definition \ref{def:E_r}) $\phi_j \simeq \varphi^d\circ \phi_i$. Let $f\colon E_{i^m}\to E_{j^m}$  be the map obtained by the universal property of pullbacks in the following homotopy commutative diagram:  \begin{equation}\label{eq:diagramas_pullback}
\begin{tikzcd}
E_{i^m}
\arrow[rr]
\arrow[ddr, bend right]
\arrow[dr, dashed, "f"]
 & &B\mathrm{SO}(2m-1) \arrow[dr, bend left=20, "\varphi^d"] \arrow[dd] & \\
& E_{j^m} \arrow[rr, crossing over] \arrow[d] \arrow[dr, phantom, "\scalebox{1.5}{$\lrcorner$}" , very near start]
& &  B\mathrm{SO}(2m-1)\arrow[d]\\
& \Sigma \arrow[r, "\phi_i"]
& B\mathrm{SO}(2m) \arrow[r, "\varphi^d"] & B\mathrm{SO}(2m)
\end{tikzcd}
\end{equation}
where the right-hand side comes from Diagram \eqref{eq:Adams_inclusions}.
Diagram \eqref{eq:diagramas_pullback} gives rise to a commutative diagram of spherical fiber sequences
\begin{equation}\label{eq:dos_fibraciones}
 \begin{tikzcd}
S^{2m-1} \arrow[r, "\widetilde{f}"] \arrow[d] & S^{2m-1} \arrow[d]\\
E_{i^m} \arrow[r, "f"] \arrow[d]  & E_{j^m} \arrow[d]\\
\Sigma \arrow[r, equal] & \Sigma,
\end{tikzcd}
\end{equation}
whose associated Serre spectral sequences (Sss) can be compared via the edge morphisms given by naturality: the Sss associated to the left (respectively, right) side of diagram \eqref{eq:dos_fibraciones} is fully determined by the differential $$\operatorname{d}_{2m}([S^{2m-1}])=\kappa i^m[\Sigma]\,\,\text{ (respectively, $\operatorname{d}_{2m}([S^{2m-1}])=\kappa j^m[\Sigma]$),}$$ and since by naturality  $$\operatorname{d}_{2m}\big(H^*(\widetilde{f})([S^{2m-1}])\big)=H^*(Id_\Sigma)\big(\operatorname{d}_{2m}([S^{2m-1}])\big)$$ we obtain that $\deg(\widetilde{f})=(j/i)^m$.

Now, the cohomological fundamental class $[E_{i^m}]$ (respectively,\ $[E_{j^m}]$) is represented by the class $[S^{2m-1}]\otimes[\Sigma]$ in the $E^{2m-1,2m}_\infty$-term of the Sss associated to the left (respectively, right) fiber sequence in diagram \eqref{eq:dos_fibraciones}. Hence by naturality
\begin{align*}
    H^*(f)([E_{j^m}])&=H^*(\widetilde{f})([S^{2m-1}])\otimes H^*(Id_\Sigma)([\Sigma])]\\
    &=\big((j/i)^m[S^{2m-1}]\big)\otimes [\Sigma]\\
    &=(j/i)^m[E_{i^m}]
\end{align*}
and therefore $\deg(f)=(j/i)^m$.
\end{proof}

\begin{remark}
Notice that manifolds $E_{r^m}$ and $E_{-r^m}$ differ in just the orientation. Hence, for every other oriented closed connected $(4m-1)$-manifold $N$,  the mapping set degree is  $D(E_{-r^m},N)=-D(E_{r^m},N)$ and $D(N,E_{-r^m})=-D(N,E_{r^m})$.
\end{remark}

Prior to proving Theorem \ref{thm:main-many-dimensions-manifolds}, we need an integral version of Proposition \ref{prop:3_rational} (see Proposition \ref{prop:iteratedsums_simply_connected}). First, we prove the following result:

\begin{lemma}\label{lem:fix.3.3..simply.conn}
Let $M,N$ be oriented closed connected $n$-manifolds, $n>2$, and $d\in D(M,N)$, $d\neq 0$. If $N$ is simply connected, then there exists a map $h\colon M\to N$ that induces a map between discs $h|_{D_{M}}\colon D_{M}\to D_{N}$ such that $D_{M}=h^{-1}(D_{N})$, $h\vert_{\partial D_{M}}\colon \partial D_{M}\to \partial D_{N}$,  and $\deg(h)=\deg(h\vert_{\partial D_{M}})=d$.
\end{lemma}

\begin{proof}
We take a map $f\colon M\to N$ of degree $d$, that we can assume to be smooth.
Take a regular value $y_0\in N$ of $f$, and let $x_1,\ldots, x_s$ be
the preimages. We take a collection of smooth regular curves $L_j$ from
$x_1$ to $x_j$, $j=2,\ldots, s$, which do not intersect each other
except at the original point $x_1$, and write $L=\bigcup\limits_{j=2}^s L_j$. 

Around each $L_j$ we are going to build a spindle as follows:
Take a parametrization $\gamma_j\colon [0,1]\to L_j$, with $\gamma_j(0)=x_1$,
$\gamma_j(1)=x_j$, and a tubular closed neighbourhood $\Gamma_j\colon [0,1]\times
D_{\epsilon}^{n-1}\to U_j\subset M$ around $L_j$. The spindle
is the subset
 $$
 S_j^\epsilon
 :=\{ \Gamma_j(t,x)\,|\,\, |x| \leq \epsilon t(1-t), t\in [0,1]\}.
 $$
Taking $\epsilon>0$ small enough, we can assume $S_j^\epsilon\cap S_{j'}^\epsilon=\{x_1\}$ for $j\ne j'$. Then, there exists a retraction 
$r\colon M\to M$ such that it is the identity on $M\smallsetminus\bigcup S_j^\epsilon$,
and sends $S_j^{\epsilon/2}$ to $L_j$ so that $r(\Gamma_j(t, t(1-t) \bar{x}))= \gamma_j(t)$ for $\bar{x}\in D_{\epsilon/2}^{n-1}$ . Indeed,
 if $\Gamma_j(t,x)\in S^\epsilon_j$ write $(t,x)=(t,t(1-t)\bar{x})$ where $\bar{x}\in D^{n-1}_{\epsilon}$, and define
 $$
 r(\Gamma_j(t, t(1-t)\bar{x}))= \left\{
  \begin{array}{ll} \gamma_j(t), & |\bar{x}|\leq \epsilon/2.  \\
  \Gamma_j(t,(\frac2\epsilon |\bar{x}| - 1) t(1-t)\bar{x}), 
  & \epsilon/2 \leq |\bar{x}| \leq \epsilon.
 \end{array}\right.
 $$
Now, let $\hat{f}= f\circ r\colon M\to N$, which has degree $d$ and
maps $S_j^{\epsilon/2}$ to the curve $R_j=f(L_j)$ in $N$.
Note that $\hat{f}^{-1}(y_0)=\{x_1,\ldots, x_s\}.$

Since $N$ is simply connected we can construct a contraction from the loop $R_j$ to $y_0$ being careful not to cross $y_0$, as follows. The curve
$R_j$ is parametrized by $\beta_j={f}\circ \gamma_j$. Take a small 
ball $B_\delta(y_0)$ and two points $0<a<b<1$ such that $\beta_j([0,a]),
\beta_j([b,1])\subset B_\delta(y_0)$ and $\beta_j(a),\beta_j(b)\in 
\partial B_\delta(y_0)$. Using that $N\smallsetminus B_\delta(y_0)$ is simply connected, we
construct a homotopy of $\beta_j|_{[a,b]}$ to a path in
$\overline{B_\delta(y_0)}$ from $\beta_j(a)$ to $\beta_j(b)$. This gives
a homotopy from $\beta_j$ to a loop  in $\bar B_\delta(y_0)$
not crossing $y_0$. Next we take a radial retraction to $y_0$, and
by juxtaposition, we get a homotopy $H_j(t,s)$ from $\beta_j(t)$ to
the constant path $H_j(t,1)=y_0$, such that $H_j(t,s)\neq y_0$,
if $s<1$ and $t\in (0,1)$.

Now we define $\tilde{f}\colon M\to N$ as follows. First $\tilde{f}=\hat{f}$
on $M\smallsetminus S_j^{\epsilon/2}$. Now on $S_j^{\epsilon/2}$, we
set 
 $$
 \tilde{f}(\Gamma_j(t,t(1-t) \bar{x})) = H_j(t, 1-2|\bar{x}|/\epsilon),
 $$
for $\bar{x}\in D^{n-1}_{\epsilon/2}$. This is continuous, because
at the boundary $|\bar{x}|=\epsilon/2$, we have
$\tilde{f} =\beta_j(t)$. We also have $\tilde{f}|_{L_j}=y_0$,
when $|\bar{x}|=0$, and these are the only possible points such that they
take value $y_0$.  Therefore
 $$
 \tilde{f}^{-1}(y_0)=\bigcup L_j =L.
 $$
Finally, we contract $L$ in $M$, resulting in  $M/L$ being homeomorphic
to $M$ since $L$ is a collection of segments with a common 
original point. Let $p\in M$ denote the point corresponding to
$L$ under the map $M \to M/L\cong M$. Then we define
$h\colon M\cong M/L \buildrel{\tilde{f}}\over{\longrightarrow} N$ 
satisfying $h^{-1}(y_0)=\{p\}$. By choosing sufficiently small $\delta$ and $\xi>0$,  we can homotope $h$ so that $D^n_\delta(p)=h^{-1}( D^n_\xi(y_0))$ and $h|_{\partial D_\delta(p)}\colon S^{n-1}\to \partial  D^n_\xi(y_0)\cong  S^{n-1}$, which has degree $d$.
\end{proof}

\begin{proposition}\label{prop:iteratedsums_simply_connected}
Let $M_i, N_i, \, i=1, \ldots, r$, be oriented closed connected $n$-manifolds,
 $n>2$, satisfying:
\begin{enumerate}
    \item $N_i$ is simply connected, $i=1, \ldots, r$;
    \item $\pi_{n-1}(N_i)=0$ for $i=1, \ldots, r$; 
    \item $D(M_i,N_j)=\{0\}$ for $i\ne j$.
\end{enumerate}
Then 
$$D(\operatornamewithlimits{\#}_{j=1}^r M_j, \operatornamewithlimits{\#}_{i=1}^r N_i)=  \bigcap_{i=1}^r D(M_i,N_i).$$  
\end{proposition}
\begin{proof}
By combining Lemmas \ref{lem:lemas_CL+NWW} and \ref{lem:lema_NWW}, it follows directly that:
 $$D(\operatornamewithlimits{\#}_{j=1}^r M_j, \operatornamewithlimits{\#}_{i=1}^r N_i)  \subset \bigcap_{i=1}^r D(\operatornamewithlimits{\#}_{j=1}^r M_j,N_i) = \bigcap_{i=1}^r D(M_i, N_i).$$

Conversely, let $d\in  D(M_1, N_1)   \cap D(M_2, N_2)$. 
By Lemma \ref{lem:fix.3.3..simply.conn}, there exists $h_i\colon M_i\to N_i$ that induces a map between discs $h_i|_{D_{M_i}}\colon D_{M_i}\to D_{N_i}$ such that $D_{M_i}=h_i^{-1}(D_{N_i})$ and $h_i\vert_{\partial D_{M_i}}\colon \partial D_{M_i}\to \partial D_{N_i}$, and $\deg(h_i)=\deg(h_i\vert_{\partial D_{M_i}})=d$, $i=1,2$. As $\deg(h_1\vert_{\partial D_{M_1}})=\deg(h_2\vert_{\partial D_{M_2}})=d$, we homotope $h_2$ so that 
$h_2\vert_{\partial D_{M_2}}=h_1\vert_{\partial D_{M_1}}$. 
Now, gluing $M_1\# M_2$ along $\partial D_{M_i}$, $i=1,2$, and $N_1\# N_2$ along $\partial D_{N_i}$, $i=1,2$, gives us a well-defined a map
$$
h_1\#h_2\colon M_1\#M_2\to N_1\#N_2
$$
whose degree, by construction,  is precisely $d$. Therefore
$$
D(M_1,N_1) \cap D(M_2,N_2)  \subset D(M_1\# M_2,N_1\#N_2).
 $$
 Now, by an inductive argument
 \begin{multline*}
 \bigcap_{i=1}^r D(M_i, N_i)=\\
 \big(\bigcap_{i=1}^{r-1} D(M_i, N_i)\big)\cap D(M_r,N_r)\subset D\big((\operatornamewithlimits{\#}_{j=1}^{r-1} M_j)\# M_r,(\operatornamewithlimits{\#}_{i=1}^{r-1} N_i)\#N_r\big)=\\
 D(\operatornamewithlimits{\#}_{j=1}^r M_j, \operatornamewithlimits{\#}_{i=1}^r N_i)
 \end{multline*}
 which concludes the proof.    
\end{proof}

\subsection*{Proof of Theorem \ref{thm:main-many-dimensions-manifolds}}
Let $\Sigma$ be an oriented closed $k$-connected $2m$-manifold  satisfying that $\Sigma_{(0)}$ is inflexible and $\pi_j(\Sigma_{(0)})=0$ for $j\geq 2m-1$.
Let $A=\{0, d_1,\ldots, d_n\}$ where $d_1,d_2,\ldots,d_n$ are pairwise different non-zero integers.

According to Proposition \ref{prop:clave-aritmetica}, there exist finite sequences $B(i)$, $i=0,\ldots, n$, of not necessarily pairwise distinct non-zero integers,  such that every element in $B(i)$ can be written as $\pm r^{m}$ for $0<r\in\Z$  with $\gcd(r, m!) = 1$, and
\begin{equation*}
A=\bigcap_{i=0}^n  S_{B(i)}.
\end{equation*}

Choose pairwise distinct prime numbers $q_0,q_1,\ldots, q_n$, in such a way that  $$q_j>\max\{|b|\,\,|\, b\in B(i), i=0,\ldots, n\}$$ and $\gcd(q_j,m!)=1$, for $j=0,\ldots, n$. Let $\alpha_i=q_i^m\prod\limits_{b\in B(i)}b$, for every $i=0,\ldots,n$. Note that $\alpha_i$ and $\alpha_i/b$, $b\in B(i)$, are integers that, up to a sign, can be written as $r^m$ for some positive integer $r$ such that $\gcd (r,m!)=1.$ Hence, following the notation in Definition \ref{def:E_r}, we can define the following $(4m-1)$-manifolds
\begin{align*}
M_i &= \operatornamewithlimits{\#}_{b\in B(i)} E_{\alpha_i/b}  \\
N_i & = E_{\alpha_i}
\end{align*}
for $i=0,\ldots,n$.
%, using that $\alpha_i$ and $\alpha_i/b$ are of the form $\pm r^m$ with $\gcd(r,m!)=1$. 
According to Lemma \ref{lem:grados_E_i} and Lemma \ref{lem:lemas_CL+NWW}, we deduce that
\begin{gather*}
D(M_i,N_i) =S_{B(i)}\, ,\,\text{and}\\
D(M_i,N_j)=\{0\},\,\text{for $i\neq j$.}
\end{gather*}

Finally, we construct
  \begin{align*}
  M&= M_0\# M_1\# \ldots \# M_n, \\
  N&= N_0\# N_1\# \ldots \# N_n.
  \end{align*}
If $k>0$, thus the $N_i$ are simply connected, according to Proposition \ref{prop:iteratedsums_simply_connected}, we obtain that
 $$
 D(M,N)=\bigcap_{i=0}^n  S_{B(i)} = A.
$$
If $k=0$, we apply Proposition \ref{prop:iteratedsums} to get that there exists $\ell\geq 0$ such that 
 $$
 D(M\# H(4m-1,\ell),N)=\bigcap_{i=0}^n  S_{B(i)} = A.
$$

\bibliographystyle{abbrv}
\bibliography{CMV-references}
\end{document}